\documentclass[a4paper,oneside,11pt]{article}
\usepackage[utf8]{inputenc}
\usepackage{lmodern,stackrel}
\usepackage[T1]{fontenc}
\usepackage{verbatim}
\usepackage{textcomp}
\usepackage[english]{babel}
\usepackage[pdftex]{hyperref}
\usepackage[pdftex]{color,graphicx}
\usepackage{amssymb}
\usepackage{amsmath}
\usepackage{upgreek}
\usepackage{subcaption}
\usepackage{epstopdf}
\usepackage{dsfont}
\usepackage{tikz}

\renewcommand{\leq}{\leqslant}
\renewcommand{\geq}{\geqslant}

\usepackage{amsmath,amsfonts,amssymb,amsthm,mathrsfs,stackrel}

\def\build#1_#2^#3{\mathrel{
\mathop{\kern 0pt#1}\limits_{#2}^{#3}}}

\newcommand{\E}{\mathbb{E}}

\theoremstyle{plain}
\newtheorem{theorem}{Theorem}

\newtheorem{corollary}{Corollary}
\newtheorem{proposition}[corollary]{Proposition}

\newtheorem{lemma}{Lemma}

\theoremstyle{definition}

\theoremstyle{remark}

\newcommand{\Z}{{\mathbb{Z}}}
\newcommand{\N}{{\mathbb{N}}}

\newcommand{\V}{{\mathbb{V}}}
\newcommand{\T}{{\mathbb{T}}}

\begin{document}
\title{The Constrained-degree percolation model}
\author{B.N.B. de Lima\footnote{Departamento de Matem{\'a}tica, Universidade Federal de Minas Gerais, Av. Ant\^onio
Carlos 6627 C.P. 702 CEP 30123-970 Belo Horizonte-MG, Brazil}, R. Sanchis$^*$, D.C. dos Santos$^*$,\\
 V. Sidoravicius\footnote{NYU-Shanghai, 1555 Century Av., Pudong Shanghai, CN 200122, China}, R. Teodoro\footnote{Departamento de Matem{\'a}tica, Universidade Federal do Rio Grande do Norte, Av. Senador Salgado Filho, s/nº
Lagoa Nova, CEP 59078-970 Natal-RN, Brazil}}

\date{}
\maketitle


\begin{abstract}
In the Constrained-degree percolation model on a graph $(\V,\E)$ there are a sequence, $(U_e)_{e\in\E}$, of i.i.d. random variables with distribution $U[0,1]$ and a positive integer $k$. Each bond $e$ tries to open at time $U_e$, it succeeds if both its end-vertices would have degrees at most $k-1$. We prove a phase transition theorem for this model on the square lattice $\mathbb{L}^2$, as well as on the d-ary regular tree. We also prove that on the square lattice the infinite cluster is unique in the supercritical phase.
\end{abstract}
{\footnotesize Keywords: phase transition; constrained degree percolation; uniqueness \\
MSC numbers:  60K35, 82B43}

\section{Introduction and main results}

 Let $G=(\V,\E)$ be an infinite, connected and bounded degree graph, that is, $\sup_{v\in\V} deg(v)<+\infty$ where $deg(v)=\#\{u\in\V;\langle v,u\rangle\in\E\}$. Let $(k_v)_{v\in\V}$ be a sequence of integers such that $k_v\leq deg(v),\forall v$ and $(U_e)_{e\in\E}$ be a sequence of {\em i.i.d.} random variables with uniform distribution in [0,1]. Define a continuous-time percolation model where at time $t=0$ all edges are closed and each edge $e=\langle v_1,v_2\rangle$ opens at time $U_e$ if $\#\{u\in\V;\langle v_i,u\rangle\mbox{ is open at time }U_e\}< k_{v_i}, i=1,2$. That is, at the random time $U_e$ the bond $e$ tries to open but it succeeds if both its endpoints belong to an open cluster where each vertex $v$ has maximum degree $k_v$. This model was introduced  in \cite{Te} and it is a simplified version of the {\em router model} created by I. Benjamini \cite{Be}. 

Variations of constrained percolation models have been studied and have relations with some statistical physics models on lattices, like the dimer model. Other types of constrained percolation models on the square lattice, where only specific configurations are allowed on the vertices, are studied in \cite{GMM} and \cite{HL}. 

Like in the model treated here, the papers \cite{GJ} and \cite{GL} also consider constrained percolation models with restrictions on the degree of each vertex, in the former the Erd\"os-R\'enyi random graph $G_{n,p}$ is studied under the conditioned on the event that all vertex degrees belong to some subset, ${\cal S}$, of the nonnegative integers; in the last paper the so-called 1-2 Model is studied, a statistical mechanics model on the hexagonal lattice where the degree of each vertex must be 1 or 2.

For a more formal description of the Constrained-degree percolation model, let $([0,1]^\E, {\cal F},P)$ be the probability space where ${\cal F}$ is the $\sigma$-algebra generated by cylinder sets of $[0,1]^\E$ and $P$ is the product of Lebesgue measures in $[0,1]$. Given a sequence $\kappa=(k_v)$ of degree restrictions and uniform random times $U=(U_e)$, we will denote by $\omega_t^{G,\kappa}(U)$ the temporal-configuration of open (or 1) and closed (or 0) bonds in $\{0,1\}^\E$ at time $t$; the status of the bond $e=\langle u_1,u_2\rangle$ in this configuration is denoted by $\omega_{t,e}^{G,\kappa}(U)$. Observe that the random variable $\omega_{t,e}^{G,\kappa}(U)$ is the product of the indicator functions of the events:

$$\{U\in[0,1]^\E;U_e<t\}$$  and  $$\{U\in[0,1]^\E;\# \{v\in\V/\{u_{3-i}\}; \omega_{U_e,\langle u_i,v\rangle }^{G,\kappa}(U)=1\}<k_{u_i}\},\ i=1,2.$$

Using the Harris construction it is straightforward to show that this model is well-defined for all $t\in[0,1]$.

As usual in percolation, given a temporal-configuration $\omega_t^{G,\kappa}(U)$, the notation $0\leftrightarrow\infty$ in $\omega_t^{G,\kappa}(U)$ means that there are infinitely many vertices connected to origin by paths of open edges at time $t$. We commit some abuse of notation denoting $\{U\in[0,1]^\E;0\leftrightarrow\infty\mbox{ in }\omega_t^{G,\kappa}(U)\}$ by $\{0\leftrightarrow\infty\mbox{ at }t\}$.

We define the probability of percolation function, $\theta^{G,\kappa}(t):[0,1]\rightarrow [0,1]$ as $\theta^{G,\kappa}(t)=P\{0\leftrightarrow\infty\mbox{ at }t\}$ and the critical time $t_c(G,\kappa)=\sup\{t\in[0,1];\theta^{G,\kappa}(t)=0\}$.

From now on, when clear from the context we will drop the superindices $G$ and $\kappa$ in the notation.

Let the graph $G$ be the square lattice $\mathbb{L}^2$ and consider the degree restriction $k_v=3,\forall v\in\Z^2$. The theorem below characterizes a phase transition for the Constrained-degree percolation model on this graph.

\begin{theorem}\label{sq}For the Constrained-degree percolation model on the square lattice, it holds that $\frac{1}{2}<t_c(\mathbb{L}^2,(3))<1$.
\end{theorem}

Part of the proof of Theorem \ref{sq} is inspired by \cite{Te}, where a weaker version of this theorem was obtained. It was shown that there is percolation at time $t=1$, $P\ a.s.$

The proposition below shows that there is no percolation when $k_v=2,\forall v\in\Z^d$, thus the degree's restriction $k_v=3$ is optimal in Theorem \ref{sq}.

\begin{proposition}\label{k2}For the Constrained-degree percolation model on the hypercubic lattice $\mathbb{L}^d,\ d\geq 2$, it holds that $\theta^{\mathbb{L}^d,(2)}(t)=0,\ \forall t\in[0,1].$
\end{proposition}

\begin{proof}

It is enough to consider the time $t=1$. Let $B_n=\{v\in\Z^d; \|v\|_1\leq n\}$ be the ball of radius $n$ centered in 0 and $\partial B_n=\{v\in\Z^d; \|v\|_1= n\}$. Define the stochastic processes $(X_n)_n$, where $X_n=\#\{v\in\partial B_n;0\leftrightarrow v\mbox{ at }t=1\}$, and the filtration  $({\cal F}_n)_n$, where ${\cal F}_n=\sigma(I_v,\ v\in B_n)$ with $I_v=I_{(0\leftrightarrow v\mbox{ at }t=1)}$. The restriction $k_v=2,\forall v\in\Z^d$ implies that $E[X_{n+1}|{\cal F}_n]\leq X_n$. That is, $(X_n,{\cal F}_n)$ is a non-negative supermartingale, as $X_n$ is an integer number it follows that $X_n\rightarrow 0$ {\em a.s.}. This concludes that $\theta^{\mathbb{L}^d,(2)}(1)=0$
\end{proof}

It is natural to ask if $t_c(\mathbb{L}^d,(k))<1$ for some $k=k(d)<2d$. We haven't any answer for this question. One interesting feature of this model is that there is no obvious monotonicity for the function $\theta^{\mathbb{L}^d,k}(t)$ in $d$. 

The proof of Theorem \ref{sq} is done in Section 2. In Section 3, we prove the uniqueness of the infinite cluster in the supercritical phase. In Section 4, we consider the Constrained-degree percolation model on regular trees.

\section{Constrained-degree percolation model on the square lattice}

\subsection{Proof of Theorem \ref{sq}}
\subsubsection{The lower bound: $\frac{1}{2}<t_c(\mathbb{L}^2,(3))$}

Let $\tilde{\omega}_{t,e}:=I_{\{U_e\leq t\}}$ be the configuration of open and closed bonds for the unrestricted model, i.e., each bond $e$ opens at time $U_e$, by the Harris-Kesten result (see \cite{Gr} or \cite{BR}), we know that the critical time is $\frac{1}{2}$. The idea is to define an intermediate model that is an essential diminishment of the unrestricted model (in the sense of \cite{AG} and \cite{BBR}). More precisely we define  the essential diminishment on a regular sublattice of $\mathbb{L}^2$ as described in Example B from page 66 of \cite{Gr}. These configurations will be denoted by $\hat{\omega}_{t,e}$ and this intermediate model can be coupled (using the same uniform random variables $(U_e)$) with the Constrained-degree percolation model, in such a way that $\hat{\omega}_{t,e}\geq \omega_{t,e}$ for all $t\in[0,1]$ and $e\in\E (\mathbb{L}^2)$.

Define the box $\Lambda=\{(x_1,x_2)\in\mathbb{Z}^2;x_1=0,1,2,3\mbox{ and }x_2=0,1,2\}$ and for each $(m,n)\in\mathbb{Z}^2$, define $\Lambda_{m,n}=(4m,3n)+\Lambda$. Let $\mathbb{E}(\Lambda_{m,n})$ be the set of bonds in $\E (\mathbb{L}^2)$ with both endpoints in $\Lambda_{m,n}$ and define $$g_{m,n}:=\langle (4m+1,3n+1),(4m+2,3n+1)\rangle,$$
$$A_{m,n}:=\{e\in\mathbb{E}(\Lambda_{m,n});\#(e\cap\partial\Lambda_{m,n})=1\}$$ and $$B_{m,n}:=\{e\in\mathbb{E}(\Lambda_{m,n});\#(e\cap\partial\Lambda_{m,n})=2\};$$ observe that $\mathbb{E}(\Lambda_{m,n})=A_{m,n}\cup B_{m,n}\cup \{g_{m,n}\}$. Finally, define the event $$C_{m,n}:=\{U\in [0,1]^{\E (\mathbb{L}^2)};\max_{e\in A_{m,n}}U_e<\min_{e\in B_{m,n}\cup\{g_{m,n}\}}U_e\}$$ and the configuration for the intermediate model as

\begin{equation}\nonumber
\hat{\omega}_{t,e}=\left\{
\begin{array}
[c]{l}%
\tilde{\omega}_{t,e},\ \ \ \ \ \ \mbox{  if } e\notin\cup_{m,n}{g_{m,n}},\\
\tilde{\omega}_{t,e}.I_{C_{m,n}^c},\mbox{  if } e =g_{m,n}.
\end{array}\right.\label{eq:truncation}
\end{equation}

Observe that the configuration $\hat{\omega}_{t,e}$ is an essential diminishment of the configuration $\tilde{\omega}_{t,e}$ for the unrestricted percolation model, they coincide except on the edges $g_{m,n}$ when the event $C_{m,n}$ occurs and $\tilde{\omega}_{t,g_{m,n}}=1$. Then, by results of \cite{AG} or \cite{BBR}, the strict inequality $\frac{1}{2}<\hat{t}_c(\mathbb{L}^2,(3))$ for the intermediate-model time threshold holds. By construction of $\omega_{t,e}$ and $\hat{\omega}_{t,e}$, it follows that $\hat{\omega}_{t,e}\geq \omega_{t,e}$ for all $t\in[0,1]$ and $e\in\E (\mathbb{L}^2)$. Thus, we can obtain the lower bound $\frac{1}{2}<t_c(\mathbb{L}^2,(3))$.

\subsubsection{The upper bound: $t_c(\mathbb{L}^2,(3))<1$}

The proof is based on a Peierls argument. Let $(\mathbb{L}^2)^*:=\mathbb{L}^2 + (\frac{1}{2},\frac{1}{2})$ be the dual lattice of  $\mathbb{L}^2$, for each bond $e\in\E(\mathbb{L}^2)$ we define its dual $e^*$ as the unique bond of $\E((\mathbb{L}^2)^*)$ that crosses $e$. Given a temporal-configuration $\omega_t$ on $\mathbb{L}^2$ (here we are omitting indices) we define the induced dual temporal-configuration $\omega^*_t$ on $( \mathbb{L}^2)^*$ declaring each bond as open if and only if its dual is open, that is, $\omega^*_{t,e^*}=\omega_{t,e}$.

A contour on $(\mathbb{L}^2)^*$ is a finite path $\langle x_0,x_1,\dots,x_l\rangle$ such that $\langle x_i,x_{i+1}\rangle\in\E((\mathbb{L}^2)^*),\forall i=0,\dots,l-1,\ x_0=x_l$ and $x_i\neq x_j$ if $|i-j|<l$. The proof is based on the topological fact that in $\mathbb{L}^2$ some cluster of open edges is finite if and only if it is surrounded by a contour of closed edges in $(\mathbb{L}^2)^*$ (see \cite{Ke} for a formal proof of this fact). 

Given a contour $\gamma=\langle x_0,x_1,\dots,x_l\rangle$ we denote by $n(\gamma)$ its number of bonds (that is, $n(\gamma)=l$), by $r(\gamma)$ its number of sides and by $m(\gamma)$ its number of sides with length one, that is, $$r(\gamma):=\#\{i\in\{0,\dots,l-1\};x_i-x_{i-1}\neq x_{i+1}-x_{i}\}$$ and $$m(\gamma):=\#\{i\in\{0,\dots,l-1\};x_i-x_{i-1}\neq x_{i+1}-x_{i}\mbox{ and }x_{i+1}-x_{i}\neq x_{i+2}-x_{i+1}\}.$$

Define $\Gamma_{n,r,m}$ as the set of all contours surrounding the origin with $n$ bonds, $r$ sides and $m$ sides with length one and define $\Gamma:=\cup_{n,r,m}\Gamma_{n,r,m}$ and $\Gamma_{n}:=\cup_{r,m}\Gamma_{n,r,m}$. 

Let $C_{n,r,m}(t):=\cup_{\gamma\in\Gamma_{n,r,m}}(\gamma\mbox{ is closed in }\omega^*_t)$ be the event that there is a closed circuit in $\Gamma_{n,r,m}$, surrounding the origin at time $t$. Analogously, we define $C_{n}(t):=\cup_{\gamma\in\Gamma_{n}}(\gamma\mbox{ is closed in }\omega^*_t)$. By Borel-Cantelli's Lemma, it is enough to prove that there exists $t<1$ such that 
\begin{equation}\label{serieBC}
\sum_{n=4}^\infty P(C_n(t))<\infty.
\end{equation}

The convergence of the series in \ref{serieBC} is based on the two lemmas below whose proofs are given in the next section.

\begin{lemma}\label{lema1}For fixed $n,r,m\in\N$, there exists a function $\epsilon:(0,1)\rightarrow (0,+\infty)$ with $\lim_{t\rightarrow 1}\epsilon (t)=0,$ such that, for all contours $\gamma\in\Gamma_{n,r,m}$, it holds that
\begin{equation}\nonumber
P(\gamma\mbox{ is closed in }\omega^*_t)\leq \left\{
\begin{array}
[c]{l}%
\big(\frac{2}{3}+\epsilon\big)^{n-2r+m}\big(p +\epsilon\big)^{2r-2m}\epsilon^{m},\ \ \ \mbox{  if } r>4\mbox{ or }m\neq 2\\
\big(\frac{1}{2}+\epsilon\big)^{n-6},\ \ \ \ \ \ \mbox{  if } r=4\mbox{ and }m=2
\end{array}\right.
\end{equation}
with $p =(\frac{439}{18144})^{\frac{1}{4}}$.
\end{lemma}

\begin{lemma}\label{lema2}For any $n,r,m\in\N$,

\[\#\Gamma_{n,r,m}\leq\frac{n^2}{4}2^{r}\binom{r}{m}\binom{n-r-1}{r-m-1}.\]
\end{lemma}

Dropping the variable $t$ in the notation, we can write the sum in (\ref{serieBC}) as

\begin{equation}\label{BC2}\nonumber
\sum_{n=4}^\infty P(C_n)=\sum_{n=4}^\infty P(C_{n,n,n})+\sum_{n\geq 4}^\infty P(C_{n,4,2})+\sum_{m\geq 0}^\infty\sum_{n\geq m+1}^{\infty}\sum_{r=m+1}^{\frac{m+n}{2}} P(C_{n,r,m}).
\end{equation}
The first term in the right-hand side counts over the set of circuits with all sides have length one, the second term counts over the set of rectangular circuits and in the third term we are using that $n\geq2r-m$. By Lemmas \ref{lema1} and \ref{lema2}, it holds that

$$\sum_{n=4}^\infty P(C_{n,n,n})\leq\sum_{n=4}^\infty n^2 2^n\epsilon^n<\infty$$ and 
$$\sum_{n\geq 4}^\infty P(C_{n,4,2})\leq\sum_{n=4}^\infty \big(\frac{1}{2}+\epsilon\big)^{n-6}<\infty$$
for some $t<1$, since $\lim_{t\rightarrow 1}\epsilon(t)=0$. Thus, to conclude the proof of Theorem \ref{sq}, it is enough to show, for some $t<1$, the convergence of

$$S(t)=\sum_{m\geq 0}^\infty\sum_{n\geq m+1}^{\infty}\sum_{r=m+1}^{\frac{m+n}{2}} P(C_{n,r,m}).$$ Using Lemmas \ref{lema1} and \ref{lema2} again, we can bound
$$S(t)\leq\sum_{m\geq 0}\bar{\epsilon}^m\sum_{n\geq m+1}n^2(\frac{2}{3}+\epsilon)^{n}\sum_{r=m+1}^{\frac{m+n}{2}}\binom{r}{m}\binom{n-r-1}{r-m-1}\alpha^{r},$$
where $\alpha=\frac{9(p+\epsilon)^2}{2}$ and $\bar{\epsilon}=\frac{(2/3+\epsilon)\epsilon}{p^2}$. Doing the change of variables $s=r-m$ and $k=n-m$, we have that
\begin{equation}\label{BC3}S(t)\leq\sum_{m\geq 0}\tilde{\epsilon}^m\sum_{k\geq 1}(k+m)^2(\frac{2}{3}+\epsilon)^{k}\binom{k+m}{m}\sum_{s=1}^{\frac{k}{2}}\binom{k-s-1}{s-1}\alpha^{s},
\end{equation} 
with $\tilde{\epsilon}=(2/3+\epsilon)\alpha\bar{\epsilon}$ (observe that $\bar{\epsilon},\tilde{\epsilon}\rightarrow 0$ as $t\rightarrow 1$). The next proposition controls the inner summation in the r.h.s. of (\ref{BC3}).

\begin{proposition}\label{prop1}For all $k\geq 2$, the function $$f(s)=\alpha^s\binom{k-s-1}{s-1},\ 1\leq s\leq\frac{k}{2}$$ reaches its maximum in $s^*=\lceil r_1\rceil$, where $r_1$ is defined in Equation \ref{raiz} below.
\end{proposition}
\begin{proof}Define $$g(s):=\frac{f(s+1)}{f(s)}-1=\frac{(4\alpha+1)s^2+[(4\alpha+1)k-(2\alpha+1)]s+ \alpha k(k-1)}{s(k-s-1)},$$ we have that $f(s+1)\geq f(s)$ if and only if $g(s)\geq 0$.

Note that the quadratic function $(4\alpha+1)s^2+[(4\alpha+1)k-(2\alpha+1)]s+\alpha k(k-1)$ has roots

\begin{equation}\label{raiz}
r_1=\frac{k}{2}-\frac{2\alpha +1 +\Delta^{1/2}}{8\alpha +2}
\end{equation} and
\begin{equation}\nonumber
r_2=\frac{k}{2}-\frac{2\alpha +1 -\Delta^{1/2}}{8\alpha +2},
\end{equation} 
where $\Delta=(4\alpha +1)k(k-2)+(2\alpha +1)^2>0$. As $r_2\geq k/2$ the maximum of $f(s)$ is reached in $\lceil r_1\rceil$. 

\end{proof}
By Proposition \ref{prop1} we can bound
\begin{equation}\label{BC4}
S(t)\leq\sum_{m\geq 0}\tilde{\epsilon}^m\sum_{k\geq 1}(k+m)^3(\frac{2}{3}+\epsilon)^{k}\binom{k+m}{m}\binom{k-s^*-1}{s^*-1}\alpha^{s^*}.
\end{equation}

Using Stirling's approximation we can find some constant $c$ such that

\begin{equation}\label{stir}\nonumber
\binom{k-s^*-1}{s^*-1}\leq c\frac{(k-s^*-1)^{k-s^*-1}}{(s^*-1)^{s^*-1}(k-2s^*)^{k-2s^*}}
\end{equation}
and we observe that

\begin{equation}\label{lim}
\lim_{k\rightarrow\infty}\big[c\frac{(k-s^*-1)^{k-s^*-1}}{(s^*-1)^{s^*-1}(k-2s^*)^{k-2s^*}}\alpha^{s^*}\big]^{\frac{1}{k}}=\frac{(1-\beta)^{1-\beta}}{\beta^{\beta}(1-2\beta)^{1-2\beta}}\alpha^{\beta},
\end{equation}
where $\beta=\lim_{k\rightarrow\infty}\frac{s^*}{k}=\frac{1}{2}-\frac{1}{2\sqrt{4\alpha +1}}$.
For the right-hand side in Equation \ref{lim}, it holds that
\begin{equation}\label{lim2}\nonumber
\frac{(1-\beta)^{1-\beta}}{\beta^{\beta}(1-2\beta)^{1-2\beta}}\alpha^{\beta}=\frac{1-\beta}{1-2\beta}=\frac{1}{2}+\frac{1}{2}\sqrt{18(p+\epsilon)^2 +1}.
\end{equation}

Then there exists a constant $c>0$ such that,
 
\begin{equation}\label{expo}
(\frac{2}{3}+\epsilon)^{k}\binom{k-s^*-1}{s^*-1}\alpha^{s^*}\leq c\delta(t)^k,\ \forall k
\end{equation}
where $\delta (t)=(\frac{2}{3}+\epsilon)[\frac{1}{2}+\frac{1}{2}\sqrt{18(p+\epsilon)^2 +1}]$ and $\delta(t)<1$ for some $t<1$ (note that $\delta(1)<0.984).$

By Inequalities \ref{BC4} and \ref{expo}, we have

\[S(t)\leq c\sum_{m\geq 0}\tilde{\epsilon}^m\sum_{k\geq 1}(k+m)^3\binom{k+m}{m}\delta(t)^k\]

\[\leq c\sum_{m\geq 0}\tilde{\epsilon}^m(m+3)^3\sum_{k\geq 1}\binom{k+m+3}{m+3}\delta(t)^k\]

\[= c\sum_{m\geq 0}\tilde{\epsilon}^m(m+3)^3\frac{1}{(1-\delta(t))^{m+4}}\]

In the equality above we are using the identity $$\sum_{k\geq 0}\binom{k+j}{j}x^k=\frac{1}{(1-x)^{j+1}},\ \forall x\in(0,1)$$ and in the last inequality we are using the bound $$(k+m)^3\binom{k+m}{m}\leq (m+3)^3\binom{k+m+3}{m+3},\ \forall k,m\in\mathbb{N}.$$ Thus, $S(t)<\infty$, for some $t<1$, given that $\lim_{t\rightarrow 1}\frac{\tilde{\epsilon}}{1-\delta(t)}=0$, finishing the proof of Theorem \ref{sq}

\subsection{Proof of Lemmas}
\subsubsection{Proof of Lemma \ref{lema1}}

Initially, we observe that at time $t=1$ there doesn't exist any closed bond $\langle x,y \rangle\in\E$, with $deg(x,\omega)\leq 2$ and $deg(y,\omega)\leq 2$. This situation can occur at $t<1$ and we would like to show that this occurs with vanishing probability when $t$ goes to one. For this purpose define the following sets of triplets of bonds in $\E^*$ (see Figure \ref{fig:animals}):

{\small \[A_1=\{\langle (-\frac{1}{2},-\frac{1}{2}),(+\frac{1}{2},-\frac{1}{2})\rangle,\langle(-\frac{1}{2},+\frac{1}{2}),(+\frac{1}{2},+\frac{1}{2})\rangle,\langle(-\frac{1}{2},+\frac{3}{2}),(+\frac{1}{2},+\frac{3}{2})\rangle\},\]
\[A_2=\{\langle(-\frac{1}{2},-\frac{1}{2}),(+\frac{1}{2},-\frac{1}{2})\rangle,\langle(-\frac{1}{2},+\frac{1}{2}),(+\frac{1}{2},+\frac{1}{2})\rangle,\langle(-\frac{1}{2},+\frac{1}{2}),(-\frac{1}{2},+\frac{3}{2})\rangle\},\]
\[A_3=\{\langle(+\frac{1}{2},-\frac{1}{2}),(+\frac{1}{2},+\frac{1}{2})\rangle,\langle(-\frac{1}{2},+\frac{1}{2}),(+\frac{1}{2},+\frac{1}{2})\rangle,\langle(-\frac{1}{2},+\frac{1}{2}),(-\frac{1}{2},+\frac{3}{2})\rangle\}.\]}

Thus, it is necessary  to take these triplets into account in the estimations for $t<1$.

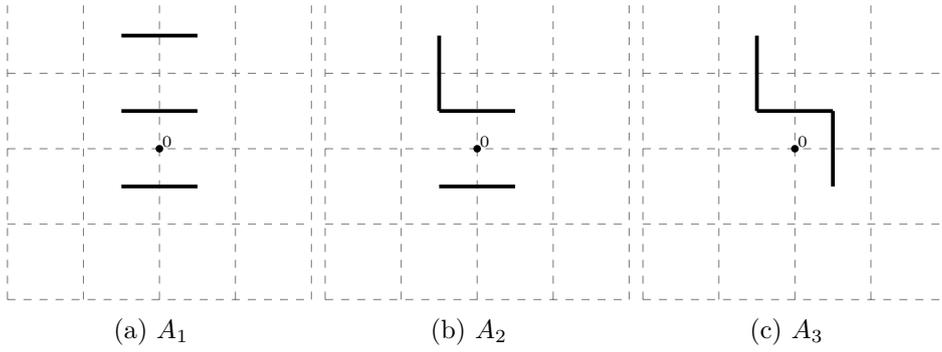
\begin{figure}
    \begin{subfigure}[b]{0.3\textwidth}
\begin{tikzpicture}
\draw[color=gray, dashed] (-2,-2) grid (2,2);
\fill (0,0) circle (0.05);
\draw(0.1,0.1)node{\tiny $0$};
\draw[line width=.05cm] (-1/2,1/2)--(1/2,1/2);
\draw[line width=.05cm] (-1/2,-1/2)--(1/2,-1/2);
\draw[line width=.05cm] (-1/2,3/2)--(1/2,3/2);
\end{tikzpicture}        
        \caption{$A_1$}
        \label{fig:gull}
    \end{subfigure}
    ~ 
    \begin{subfigure}[b]{0.3\textwidth}
    \begin{tikzpicture}
\draw[color=gray, dashed] (-2,-2) grid (2,2);
\fill (0,0) circle (0.05);
\draw(0.1,0.1)node{\tiny $0$};
\draw[line width=.05cm] (-1/2,1/2)--(1/2,1/2);
\draw[line width=.05cm] (-1/2,-1/2)--(1/2,-1/2);
\draw[line width=.05cm] (-1/2,1/2)--(-1/2,3/2);
\end{tikzpicture}  
        \caption{$A_2$}
        \label{fig:tiger}
    \end{subfigure}
    ~ 
    \begin{subfigure}[b]{0.3\textwidth}
    \begin{tikzpicture}
\draw[color=gray, dashed] (-2,-2) grid (2,2);
\fill (0,0) circle (0.05);
\draw(0.1,0.1)node{\tiny $0$};
\draw[line width=.05cm] (-1/2,1/2)--(1/2,1/2);
\draw[line width=.05cm] (-1/2,1/2)--(-1/2,3/2);
\draw[line width=.05cm] (1/2,1/2)--(1/2,-1/2);
\end{tikzpicture}  
        \caption{$A_3$}
        \label{fig:mouse}
    \end{subfigure}
    \caption{The sets $A_1,A_2$ and $A_3$}\label{fig:animals}
\end{figure} 
 
For all of these three sets, we define the central bond 
$$c(A_i)=\langle(-\frac{1}{2},+\frac{1}{2}),(+\frac{1}{2},+\frac{1}{2})\rangle,\ i=1,2,3.$$ Consider the set ${\mathscr A}=\{f(A_i); f:\Z^2\rightarrow\Z^2\mbox{ isometry and }i=1,2,3\}$, one element of ${\mathscr A}$ is called an {\em anomaly} and given an anomaly $f(A_i)$, we define its center as the edge $f(c(A_i))\in\E^*$.

Given a contour $\gamma=\langle x_0,x_1,\dots,x_l\rangle$ in $(\mathbb{L}^2)^*$, let $\E (\gamma)$ be the set of bonds of $\gamma$, an edge $e^*\in\E (\gamma)$ is called {\em anomalous} if there exists some $A\in{\mathscr A}$ such that $e^*\in A$ and $A\subset\E (\gamma)$. Let $A(\gamma)$ and $U(\gamma)$ be the set of anomalous and unitary bonds, respectively. That is $A(\gamma)=\{e^*\in\E (\gamma);e^*\mbox{ is anomalous}\}$ and $U(\gamma)=\{e^*=\langle x_i,x_{i+1}\rangle\in\E (\gamma);\ x_{i-1}-x_{i}\neq x_i-x_{i+1}\mbox{ and } x_i-x_{i+1}\neq x_{i+1}-x_{i+2}\}$ (observe that $\#U(\gamma)=m(\gamma)$).

\begin{proposition}\label{anom} For all $\gamma\in\Gamma$ with $r(\gamma)>4$ or $m(\gamma)\neq 2$, it holds that $\#A(\gamma)\geq \#U(\gamma)$.
\end{proposition}
\begin{proof} It is enough to exhibit an injective function $F:U(\gamma)\rightarrow A(\gamma)$. If $e^*\in U(\gamma)\cap A(\gamma)$, define $F(e^*)=e^*$. If $e^*\in U(\gamma)\cap A(\gamma)^c$, suppose that $e^*=\langle x_i,x_{i+1}\rangle$ where $x_{i+1}=x_i+e_j$ with $j=1$ or 2, as $e^*\notin A(\gamma)$ it holds that $x_{i+2}-x_{i+1}=x_{i-1}-x_{i}=e_{3-j}$ or $x_{i+2}-x_{i+1}=x_{i-1}-x_{i}=-e_{3-j}$. Let us consider the first case, the second one is analogous.

Define integers $l^\prime(e^*)$ and $l^{\prime\prime} (e^*)$ as

$$l^\prime(e^*)=\min\{l\geq 1;\ x_{i+l+1}-x_{i+l}\neq e_{3-j}\}$$
and
$$l^{\prime\prime}(e^*)=\min\{l\geq 1;\ x_{i-l}-x_{i-l+1}\neq e_{3-j}\}.$$
Observe that $l^\prime(e^*),l^{\prime\prime}(e^*)\geq 3$. Thus, for $e^*\in U(\gamma)\cap A(\gamma)^c$, we define $F(e^*)$ as:
\begin{equation}\nonumber
F(e^*)=\left\{
\begin{array}
[c]{l}%
\langle x_{i+l^\prime(e^*)-1},x_{i+l^\prime(e^*)}\rangle,\ \ \ \ \mbox{  if } l^\prime(e^*)\leq l^{\prime\prime}(e^*),\\
\langle x_{i-l^{\prime\prime}(e^*)+1},x_{i-l^{\prime\prime}(e^*)+2}\rangle,\mbox{  if } l^\prime(e^*)> l^{\prime\prime}(e^*).
\end{array}\right.
\end{equation}

The function $F$ above is the desired injective function.
\end{proof}

Observe that Proposition \ref{anom} does not hold for rectangular circuits with sides with length one ($r(\gamma)=4$ and $m(\gamma)=2$). Thus, from now on we will deal only with circuits $\gamma$ where $r(\gamma)>4$ or $m(\gamma)\neq 2$, then Proposition \ref{anom} can be used. The case $r(\gamma)=4$ and $m(\gamma)= 2$ is postponed to the end of this subsection.

We define $B(\gamma)$ and $C(\gamma)$ as the following subsets of $\E (\gamma)$:  
\[B(\gamma)=\{e^*=\langle x_i,x_{i+1}\rangle;x_{i-1}-x_{i}=x_i-x_{i+1}=x_{i+1}-x_{i+2}\}\cap A(\gamma)^c\] and 
\[C(\gamma)=\left\{\begin{array}{l} e^*=\langle x_i,x_{i+1}\rangle;x_{i-1}-x_{i}\neq x_i-x_{i+1}\\\mbox{or } x_i-x_{i+1}\neq x_{i+1}-x_{i+2}\end{array}\right\}\\ \cap (A(\gamma)\cup U(\gamma))^c.\]

Indeed, we will prove the stronger bound:

\begin{equation}\label{abc}
P(\gamma\mbox{ is closed in }\omega^*_t)\leq \big(\frac{2}{3}+\epsilon\big)^{\#B}\big(p +\epsilon\big)^{\#C}\epsilon^{\#A}.
\end{equation} 

By Proposition \ref{anom}, $\#A\geq m$ and observing that $\#B\leq n-2r+m$, $\#C\leq 2r-2m$ (because the definitions of $B$ and $C$ exclude the anomalous bonds) and $(\frac{2}{3}+\epsilon)\geq(p +\epsilon)\geq\epsilon$, the statement of Lemma \ref{lema1} follows from Inequality \ref{abc}.

To prove the inequality \ref{abc}, we consider each bond $e^*\in \E(\gamma)$ and bound its probability to be closed at time $t$ depending if $e^*$ belongs to $A(\gamma), B(\gamma)$ or $C(\gamma)$.

If $e^*\in A(\gamma)$, choose one of the possible anomalies that contains $e^*$ and define $c(e^*)$ as the dual bond of its center. By definition of anomaly $U(c(e^*))>t$ and $$P(e^*\mbox{ is closed in }\omega^*_t)\leq P(U(c(e^*))>t)=1-t,\forall e^*\in A(\gamma).$$ Then,

\begin{equation}\label{anomalos}\nonumber P(e^*\mbox{ is closed in }\omega^*_t,\forall e^*\in A(\gamma) )\leq (1-t)^{\lfloor\frac{\#A}{9}\rfloor}.
\end{equation}

Let $X(e^*)\in\E$ be the set $\{c(e^*)\},\forall e^*\in A(\gamma)$. Given $e^*\in B(\gamma)\cup C(\gamma)$ we define the set of bonds $X(e^*)\subset\E$. Let us suppose that $\langle x_0,x_1,\dots,x_l\rangle$ is the side of $\gamma$ that contains $e^*$ and let $x_{-1}$ and $x_{l+1}$ the vertices adjacent, in $\gamma$, to $x_0$ and $x_l$, respectively. Then, $x_{j+1}-x_j=e_i,\ x_{0}-x_{-1}=\pm e_{3-i}$ and $x_{l+1}-x_l=\pm e_{3-i}$, for $i=1$ or 2. Define 

\begin{equation}\nonumber
\sigma_+=\left\{
\begin{array}
[c]{l}%
+1,\ \ \mbox{  if } x_{l+1}-x_l=e_1\mbox{ or }e_2,\\
-1,\ \ \mbox{  if } x_{l+1}-x_l=-e_1\mbox{ or }-e_2.
\end{array}\right.
\end{equation}
and
\begin{equation}\nonumber
\sigma_-=\left\{
\begin{array}
[c]{l}%
+1,\ \ \mbox{  if } x_{0}-x_{-1}=e_1\mbox{ or }e_2,\\
-1,\ \ \mbox{  if } x_{0}-x_{-1}=-e_1\mbox{ or }-e_2.
\end{array}\right.
\end{equation}

For any $e^*=\langle x_i, x_i+e_1\rangle\in B(\gamma)$ a horizontal bond, we define the set $X(e^*)$ as:

{\footnotesize \begin{equation}\nonumber
X(e^*)=\left\{
\begin{array}
[c]{l}%
\{e^*+(\frac{1}{2},\frac{\sigma_+}{2}),e^*+(-\frac{1}{2},-\frac{\sigma_+}{2})\},\mbox{  if } e^*+(0,\sigma_+)\notin\gamma\mbox{ and }e^*+(0,-\sigma_+)\notin\gamma,\\
\{e^*+(-\frac{1}{2},-\frac{\sigma_+}{2})\},\ \ \ \ \ \ \ \ \ \ \ \ \ \ \ \ \ \ \mbox{  if } e^*+(0,\sigma_+)\in\gamma\mbox{ and }e^*+(0,-\sigma_+)\notin\gamma,\\
\{e^*+(\frac{1}{2},\frac{\sigma_+}{2})\},\ \ \ \ \ \ \ \ \ \ \ \ \ \ \ \ \ \ \ \ \ \ \mbox{  if } e^*+(0,\sigma_+)\notin\gamma\mbox{ and }e^*+(0,-\sigma_+)\in\gamma.
\end{array}\right.
\end{equation}}

Analogously, if $e^*=\langle x_i, x_i+e_2\rangle\in B(\gamma)$ is a vertical bond, the set $X(e^*)\subset\E$ is defined as:

{\footnotesize \begin{equation}\label{Xinterno}
X(e^*)=\left\{
\begin{array}
[c]{l}%
\{e^*+(\frac{\sigma_+}{2},\frac{1}{2}),e^*+(-\frac{\sigma_+}{2},-\frac{1}{2})\},\mbox{  if } e^*+(\sigma_+,0)\notin\gamma\mbox{ and }e^*+(-\sigma_+,0)\notin\gamma,\\
\{e^*+(-\frac{\sigma_+}{2},-\frac{1}{2})\},\ \ \ \ \ \ \ \ \ \ \ \ \ \ \ \ \ \mbox{  if } e^*+(\sigma_+,0)\in\gamma\mbox{ and }e^*+(-\sigma_+,0)\notin\gamma,\\
\{e^*+(\frac{\sigma_+}{2},\frac{1}{2})\},\ \ \ \ \ \ \ \ \ \ \ \ \ \ \ \ \ \ \ \ \ \ \mbox{  if } e^*+(\sigma_+,0)\notin\gamma\mbox{ and }e^*+(-\sigma_+,0)\in\gamma.
\end{array}\right.
\end{equation}}

Observe that 
\begin{align*}(e^*\mbox{ is closed in }\omega^*_t)\subset (U(e)>t)\cup [(U(e)\leq t)\cap (U(e)>\min_{f\in X(e^*)}U(f))],
\end{align*}
yielding the bound
\begin{equation}\nonumber\label{internos}P(e^*\mbox{ is closed in }\omega^*_t)\leq \frac{2}{3}+(1-t),\ \forall e^*\in B(\gamma)
\end{equation}
and it holds that $X(e^*)\cap X(f^*)=\emptyset,\forall e^*\neq f^*\in A(\gamma)\cup B(\gamma)$.
 
Finally we consider the case $e^*=\langle x_i, x_i+e_j\rangle\in C(\gamma)$, there are only four different cases depending if $i=0$ 
 or $l-1$ , and $j=1$ or 2 (horizontal or vertical). Thus, for $e^*=\langle x_i, x_i+e_j\rangle\in C(\gamma)$, we define $X(e^*)$ as:

\begin{equation}\label{quina}
X(e^*)=\left\{
\begin{array}
[c]{l}%
\{e+(0,-\sigma_+),e^*+(\pm\frac{1}{2},-\frac{\sigma_+}{2})\},\ \ \mbox{  if } i=l-1\mbox{ and }j=1,\\
\{e+(0,\sigma_-),e^*+(-\frac{1}{2},\frac{\sigma_-}{2})\},\ \ \ \ \ \ \ \mbox{  if } i=0\mbox{ and }j=1,\\
\{e+(-\sigma_+,0),e^*+(-\frac{\sigma_+}{2},\pm\frac{1}{2})\},\ \ \mbox{  if } i=l-1\mbox{ and }j=2,\\
\{e+(\sigma_-,0),e^*+(\frac{\sigma_-}{2},-\frac{1}{2})\},\ \ \ \ \ \ \ \mbox{  if } i=0\mbox{ and }j=2.
\end{array}\right.
\end{equation}

With these definitions, observe that  
\begin{equation}\label{soz}
P(e^*\mbox{ is closed in }\omega^*_t)\leq \frac{1}{3}+(1-t),\ \forall e^*\in C(\gamma)
\end{equation}
and $X(e^*)\cap X(f^*)=\emptyset,\forall e^*\in C(\gamma), f^*\in A(\gamma)\cup B(\gamma)$. Then
\begin{equation}\label{cotanaoquina}P\left(\cap_{e^*\in A\cup B} (e^*\mbox{ is closed in }\omega^*_t )\right)\leq (1-t)^{\lfloor\frac{\#A}{9}\rfloor}.((1-t)+\frac{2}{3})^{\#B}
\end{equation}

We are interested in estimating the probability  $$P\left(U(e)>\min_{g\in X(e^*)}U(g),\ U(f)>\min_{g\in X(f^*)}U(g)\right),$$ unfortunately $X(e^*)\cap X(f^*)\neq\emptyset$ for some $e^*, f^*\in C(\gamma)$, thus we need a finer analysis. For any $e^*\in C(\gamma)$, the unique bond in $X(e^*)$ (see Equation \ref{quina}) parallel to $e$ is called {\em frontal}, the other one or two bonds in $X(e^*)$ parallel to $e^*$ are called {\em lateral}. 

First, let us consider the case where $X(e^*)$ and $X(f^*)$ share the same frontal bond, without loss of generality, suppose that $e^*$ is a vertical bond, $f^*=e^*+(2,0)$ and $X(e^*),X(f^*)$ share the frontal bond $e+(1,0)$, the other situations are analogous. For this case there are four different possibilities for the pair of sets $(X(e^*),X(f^*))$ depending on whether $e^*$ and $f^*$ are the lowest or the highest bond of their sides. In the worst possibility (greatest upper bound for $P(e^*,f^*\mbox{ are closed in }\omega^*_t)$), we have $X(e^*)=\{e+(1,0),e^*+(\frac{1}{2},-\frac{1}{2})\}$ and $X(f^*)=\{e+(1,0),e^*+(\frac{3}{2},-\frac{1}{2})\}$ (see Figure \ref{2sobre15}), then $P(U(e)>\min_{g\in X(e^*)}U(g),\ U(f)>\min_{g\in X(f^*)}U(g))=\frac{2}{15}$, yielding the bound

\begin{equation}\label{inter2}
\begin{array}
[c]{l}%
P(e^*,f^*\mbox{ are closed in }\omega^*_t)\leq P(U(e)>t\mbox{ or }U(f)>t)\\
+\ P\{(U(e)\leq t,U(f)\leq t)\cap (U(j)>\min_{g\in X(j^*)}U(g),\ j=e,f)\}\\
\leq\ 1-t^2 +\frac{2}{15}\ \leq \left((1-t^2)^{\frac{1}{2}}+\sqrt{\frac{2}{15}}\right)^2
\end{array}
\end{equation}

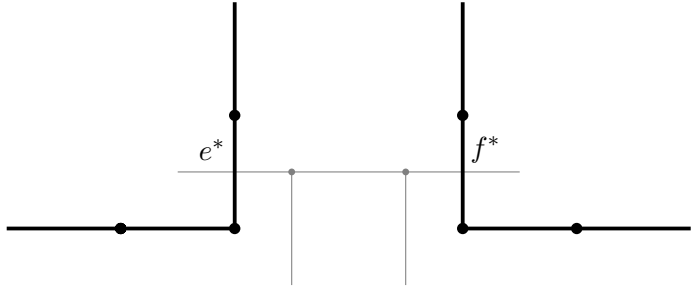
\begin{figure}
\centering
\begin{tikzpicture}[scale=1.5]
\draw[gray](1/2,1/2)--(1/2+1,1/2);
\draw[gray](1/2,1/2)--(1/2-1,1/2);
\draw[gray](1/2-1,1/2)--(1/2-2,1/2);    
\draw[line width=.05cm] (-3,0)--(-2,0)--(-1,0)--(-1,1)--(-1,2);
\fill (-2,0) circle (0.05 cm) circle (0.05 cm) (-1,0)circle (0.05 cm) (-1,1)circle (0.05 cm);
\draw[line width=.05cm](1,2)--(1,1)--(1,0)--(2,0)--(3,0);
\fill (1,1) circle (0.05 cm) (1,0) circle (0.05 cm) (2,0) circle (0.05 cm);
\fill[gray] (1/2,1/2) circle (0.03 cm);
\fill[gray] (-1/2,1/2) circle (0.03 cm);
\draw(1/2+1-0.3,1/2+0.2)node{$f^*$};
\draw(1/2-2+0.3,1/2+0.2)node{$e^*$};
\draw[gray](1/2,1/2)--(1/2,1/2-1);
\draw[gray](1/2-1,1/2)--(1/2-1,1/2-1);
\end{tikzpicture}\caption{The black edges belong to the circuit $\gamma$ and the gray edges correspond to the bonds $e,f,X(e^*)$ and $X(f^*)$.\label{2sobre15}}
\end{figure}

Now consider the case where $X(e^*)$ and $X(f^*)$ share the same lateral bond. In this situation, there exist four bonds $e^*,f^*,g^*,h^*\in C(\gamma)$, such that their sets $X(\cdot)$ have some intersections. Apart from rotations (without loss of generality suppose $e^*$ a vertical bond), there are two possible situations. In the first case, let us call it case I, we have that $f^*=e^*+(1,-1), g^*= e+(\frac{3}{2},-\frac{1}{2})$ and $h^*=e+(-\frac{1}{2},-\frac{1}{2})$, thus 

\begin{equation}\nonumber
\begin{array}
[c]{l}%
X(e^*)=\{e+(1,0),e^*+(\frac{1}{2},-\frac{1}{2})\},\\
X(f^*)=\{e+(0,-1), e^*+(\frac{1}{2},-\frac{1}{2}),e^*+(\frac{1}{2},-\frac{3}{2})\},\\
X(g^*)=\{e^*+(\frac{3}{2},\frac{1}{2}),e+(1,0)\}\\
X(h^*)=\{e^*+(-\frac{1}{2},-\frac{3}{2}),e+(0,-1),e+(-1,-1))\}.
\end{array}
\end{equation}

In the second case, case II, we have that $f^*=e^*+(1,1), g^*= e+(\frac{3}{2},\frac{1}{2})$ and $h^*=e+(-\frac{1}{2},\frac{1}{2})$, thus

\begin{equation}\nonumber
\begin{array}
[c]{l}%
X(e^*)=\{e+(1,0),e^*+(\frac{1}{2},-\frac{1}{2})e^*+(\frac{1}{2},\frac{1}{2})\},\\
X(f^*)=\{e+(0,1), e^*+(\frac{1}{2},\frac{1}{2})\},\\
X(g^*)=\{e^*+(\frac{3}{2},-\frac{1}{2}),e+(1,0)\}\\
X(h^*)=\{e^*+(-\frac{1}{2},\frac{3}{2}),e+(0,1),e+(-1,1)\}.
\end{array}
\end{equation}

Finally, when there exist two different bonds $e^*,g^*\in C(\gamma)$ with $X(e^*)$ and $X(g^*)$ sharing a edge frontal for $e^*$ and lateral for $g^*$ occurs exactly (apart rotations, without loss of generality we are supposing that $e^*$ is a vertical bond) the two cases I and II described above.

In both cases we are interested in estimating the probability of the event 

\begin{equation}\label{inter4}\bigcap_{i\in\{e,f,g,h\}}(U(i)>\min_{j\in X(i^*)}U(j)).\end{equation} 

We observe that it is possible to find some $i^*\in C(\gamma)/\{e^*,f^*,g^*,h^*\}$ such that $X(i^*)\cap \left(\cup_{j\in \{e,f,g,h\}}X(j^*)\right)\neq\emptyset$. More precisely the frontal bonds in $X(g^*)$ and $X(h^*)$ could be the frontal bond of $X(i^*)$ for some $i^*\in C(\gamma)/\{e^*,f^*,g^*,h^*\}$. To hold some independence we bound the probability of \ref{inter4} deleting the frontal bonds of $X(g^*)$ and $X(h^*)$, that is, consider the event
{\small \begin{align*}
R_{e^*,f^*,g^*,h^*}:=\left(\bigcap_{i\in\{e,f\}}(U(i)>\min_{j\in X(i^*)}U(j))\right)\bigcap\left(\bigcap_{i\in\{g,h\}}(U(i)>\min_{j\in \tilde{X}(i^*)}U(j))\right),
\end{align*}}
where $\tilde{X}(i^*)$ is the set of lateral bonds in $X(i^*)$. The probabilities for the event $R_{e^*,f^*,g^*,h^*}$ are $\frac{439}{18144}$ and $\frac{289}{12096}$ (these numbers can be calculated by hand with some effort, listing the favorable cases among the $9!$ permutations involved) for the cases I and II, respectively. Then regardless of case I or II, it holds that

\begin{equation}\label{newinter4}
\begin{array}
[c]{l}%
P(e^*,f^*,g^*,h^*\mbox{ are closed in }\omega^*_t)\leq P(\cup_{j\in \{e,f,g,h\}} (U(j)>t)) +\\
P\{(\cap_{j\in \{e,f,g,h\}} (U(j)\leq t))\cap R_{e^*,f^*,g^*,h^*}\}\ \leq
1-t^4 +\frac{439}{18144}\\ \leq \left((1-t^4)^{\frac{1}{4}}+p\right)^4,
\end{array}
\end{equation}
where $p=(\frac{439}{18144})^{\frac{1}{4}}$.

\begin{figure}
    \begin{subfigure}[b]{0.3\textwidth}
\begin{tikzpicture}[scale=1.5]
\draw[line width=.05cm] (-1.2,0)--(-1,0)--(0,0)--(0,1)--(0,1.3);
\draw[line width=.05cm] (1,-1.2)--(1,-1)--(1,0)--(2,0)--(2.2,0);
\draw[gray](-3/2,-1/2)--(-1/2,-1/2)--(1/2,-1/2)--(1/2,1/2)--(3/2,1/2)--(3/2,1/2)--(3/2,-1/2)--(1/2,-1/2)--(1/2,-3/2);
\draw[gray](-1/2,-1/2)--(-1/2,1/2)--(1/2,1/2);
\draw[dashed,gray](-1/2,-1/2)--(-1/2,-3/2);
\draw[dashed,gray](3/2,1/2)--(3/2,3/2);
\draw(1/2-0.4,1/2+0.1)node{$e^*$};
\draw(-1/2+0.1,-1/2+0.4)node{$h^*$};
\draw(1/2+0.4,-1/2+0.1)node{$f^*$};
\draw(1+0.4,0+0.1)node{$g^*$};
\fill  (-1,0) circle (0.05 cm)(-1,0) circle (0.05 cm)(0,0)
 circle (0.05 cm) (0,1)  circle (0.05 cm)  (1,-1) circle (0.05 cm) (1,0) circle (0.05 cm) (2,0) circle (0.05 cm);
\fill[gray] (-1/2,1/2) circle (0.03 cm) (1/2,1/2) circle (0.03 cm) (3/2,1/2) circle (0.03 cm)
(1/2,-1/2) circle (0.03 cm) (-1/2,-1/2) circle (0.03 cm) (3/2,-1/2) circle (0.03 cm);
\end{tikzpicture}\caption{Case I\label{beijinho}}
           \end{subfigure}
          ~~~~~~~~~~~~~~~~~~~~~~~~~~~~~~ 
\begin{subfigure}[b]{0.3\textwidth}
  \begin{tikzpicture}[scale=1.5]
\draw[line width=.05cm] (-1.2,0)--(-1,0)--(0,0)--(0,-1)--(0,-1.3);
\draw[line width=.05cm] (1,1.2)--(1,1)--(1,0)--(2,0)--(2.2,0);
\draw[gray](-3/2,1/2)--(-1/2,1/2)--(1/2,1/2)--(3/2,1/2)--(3/2,-1/2)--(1/2,-1/2)--(-1/2,-1/2)--(-1/2,1/2);
\draw[gray](1/2,1/2)--(1/2,-1/2)--(1/2,-3/2);
\draw[dashed,gray](-1/2,1/2)--(-1/2,3/2);
\draw[dashed,gray](3/2,-1/2)--(3/2,-3/2);
\draw(1/2+1-0.4,1/2+0.1)node{$f^*$};
\draw(-1/2+0.6,-1/2+0.1)node{$e^*$};
\draw(3/2+0.1,-1/2+0.4)node{$g^*$};
\draw(0-0.4,0.1)node{$h^*$};
\fill  (-1,0) circle (0.05 cm)(-1,0) circle (0.05 cm)(0,0)
 circle (0.05 cm) (0,-1)  circle (0.05 cm)  (1,1) circle (0.05 cm) (1,0) circle (0.05 cm) (2,0) circle (0.05 cm);
\fill[gray] (-1/2,1/2) circle (0.03 cm) (1/2,1/2) circle (0.03 cm) (3/2,1/2) circle (0.03 cm)
(1/2,-1/2) circle (0.03 cm) (-1/2,-1/2) circle (0.03 cm) (3/2,-1/2) circle (0.03 cm);
\end{tikzpicture}\caption{Case II\label{beijinho}}
    \end{subfigure}
    \caption{The black bonds belong to the circuit $\gamma$ and the gray bonds correspond to the bonds $e,f,g,h,X(e^*),X(f^*),X(g^*)$ and $X(h^*)$. The dotted bonds are the deleted frontal bonds in $X(g^*)$ and $X(h^*)$.}
\end{figure}
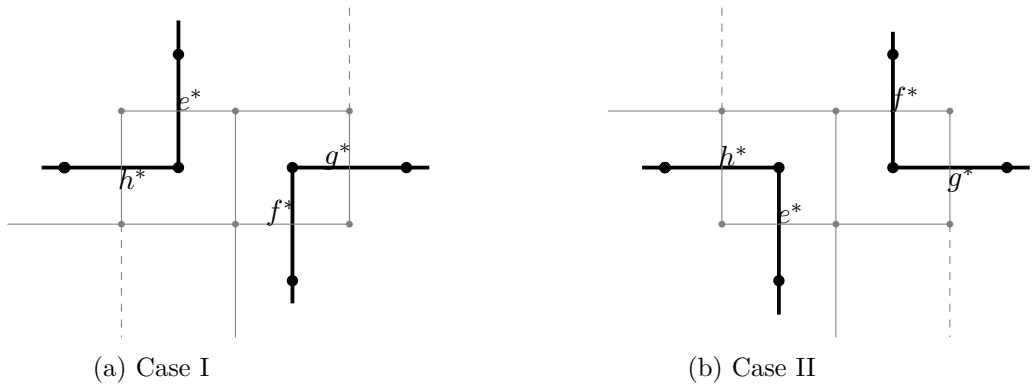

By Inequalities \ref{soz}, \ref{inter2} and \ref{newinter4}, it holds that
\begin{equation}\label{cotaquina}
P(e^*\mbox{ is closed in }\omega^*_t,\forall e^*\in C(\gamma) )\leq \left((1-t^4)^{\frac{1}{4}}+p\right)^{\#C}.
\end{equation}

Remembering that $X(e^*)\cap X(f^*)=\emptyset,\forall e^*\in C(\gamma), f^*\in A(\gamma)\cup B(\gamma)$, by Inequalities \ref{cotanaoquina} and \ref{cotaquina}, it holds that
\begin{equation}\nonumber
\begin{array}{l}P(\gamma\mbox{ is closed in }\omega^*_t)\leq P\left(\cap_{e^*\in A\cup B\cup C}(e^*\mbox{ is closed in }\omega^*_t)\right)\\
\leq \big(\frac{2}{3}+\epsilon\big)^{\#B}\big(p +\epsilon\big)^{\#C}\epsilon^{\#A},
\end{array}
\end{equation} 
where $\epsilon=(1-t)^{\frac{1}{9}}$. This proves Inequality \ref{abc} and thus concludes the proof of Lemma \ref{lema1} for circuits $\gamma$ with $r(\gamma)>4$ or $m(\gamma)\neq 2$. 

The case where the circuit $\gamma$ is in such a way $r(\gamma)=4$ and $m(\gamma)= 2$, we have that $\#B(\gamma)=n(\gamma)-6$ and by definition of $X(e^*)$ for $e^*\in B(\gamma)$ (see Equation (\ref{Xinterno})), it holds that $\#X(e^*)=1$, then 
$$P(e^*\mbox{ is closed in }\omega^*_t)\leq \left(\frac{1}{2}+(1-t)\right),\forall e^*\in B(\gamma),$$ implying that

$$P(\gamma\mbox{ is closed in }\omega^*_t)\leq P(e^*\mbox{ is closed in }\omega^*_t,\forall e^*\in B(\gamma) )\leq \left(\frac{1}{2}+(1-t)\right)^{n(\gamma)-6}.$$

\begin{figure}
\centering
\begin{tikzpicture}[scale=1.5]
\draw(0.8,1)node[above]{$e^*$};
\draw(0,3/2)node[above]{$X(e^*)$};
\draw[line width=.05cm] (-3,1)--(-3,0)--(-2,0)--(-1,0)--(0,0)--(1,0);
\draw[line width=.05cm](1,0)--(2,0)--(3,0)--(3,1);
\draw[line width=.05cm](-3,1)--(-2,1)--(-1,1)--(0,1)--(1,1)--(2,1)--(3,1);
\fill (-2,0) circle (0.05 cm) (-1,0)circle (0.05 cm)  (-1,1)circle (0.05 cm)(0,0)circle(0.05cm)(1,0)circle(0.05cm);
\fill (1,1) circle (0.05 cm) (1,0) circle (0.05 cm) (2,0) circle (0.05 cm)(3,0) circle (0.05 cm)(-3,0) circle (0.05 cm);
\fill (-3,1)circle (0.05) (-2,1) circle (0.05) (0,1) circle (0.05 cm) (1,1) circle (0.05 cm) (2,1) circle (0.05 cm)(3,1) circle (0.05 cm)(-3,1) circle (0.05 cm);
\draw[gray](1/2,1/2)--(1/2,1/2+1);
\draw[gray](1/2,3/2)--(-1/2,3/2);
\fill[gray] (1/2,3/2) circle (0.03 cm);
\end{tikzpicture}\caption{The black edges belong to the circuit $\gamma$ and the gray edges correspond to the bonds $e$ and $X(e^*)$.
If $e^*$ is closed then $U(e)>\max_{f\in X(e^*)}{U(f)}$. Therefore $P(e^* \mbox{ is closed} )\leq\frac{1}{2}+(1-t)$.
\label{2sobre15}}
\end{figure}
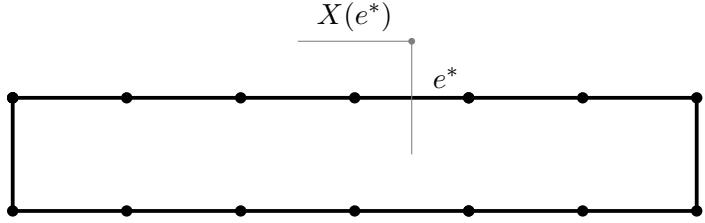

This finishes the proof of Lemma \ref{lema1}.

\subsubsection{Proof of Lemma \ref{lema2}}

Given any finite path $\gamma=\langle x_0,x_1,\dots,x_n\rangle$, let $x^*(\gamma)$ be the lowest vertex in $\gamma$ in the lexicographical order, then $\#\{x^*(\gamma); \gamma\in\Gamma_{n,r,m}\}\leq \frac{n^2}{4}$. 

To start a path $\gamma$ from $x^*(\gamma)$, we have $\binom{r}{m}$ options to choose which sides will have length one, the number of options to choose the lengths of the other sides is $\binom{n-r-1}{r-m-1}$ (the number of integer solutions of $s_1+\dots+s_{r-m}=n-m,\ s_i\geq 2,\forall i$) and after each side there are at most 2 options for the direction of the next one. Note that we are counting paths that are not circuits as well as paths with self-intersections.

\section{Uniqueness of the infinite cluster in the supercritical phase}

The study of the number of infinite clusters in percolation theory was initiated by Harris \cite{Ha}, he showed that, in the supercritical phase, the Bernoulli percolation model on $\mathbb{L}^2$ has only one infinite cluster. In 1981, Newman and Schulman \cite{NS} defined the concept of the {\it finite-energy condition} for a percolation process and showed that under this condition the number of infinite clusters is, almost surely, 0, 1 or $\infty$. In 1987,  Aizenman, Kesten and Newman \cite{AKN} excluded the possibility that there are $\infty$ infinite clusters; Burton and Keane gave a very elegant proof of this fact in the paper \cite{BK}, they also use the finite-energy property. Finally, in 1992, Gandolfi, Keane and Newman \cite{GKN} generalized the result of Burton and Keane for processes that have {\it finite positive energy}, today this property is called {\it insertion tolerance} (see \cite{BLPS} or Chapter 7 of \cite{LP}). For more about uniqueness and non-uniqueness of infinite clusters in percolation theory see the survey \cite{HJ}.

The Constrained-degree percolation model doesn't have the insertion tolerance property, nevertheless we can prove the uniqueness of the infinite cluster in the supercritical phase in some situations, as in the following theorem:

\begin{theorem}\label{unico}
Consider the Constrained-degree percolation process on $\mathbb{L}^2$ with restriction $\kappa(v)=3,\ \forall v\in\mathbb{Z}^2$. If $P\{0\leftrightarrow\infty \text{ in } t\}>0$ then $\omega_t$ has, $P$-a.s., exactly one infinite open cluster.
\end{theorem}

\begin{proof}
Let $N_t$ be the number of infinite clusters in the temporal-configuration $\omega_t$, as the measure $P$ is translation-invariant and ergodic and the random variable $N_t$ is also translation invariant, it holds that, $P$-a.s., $N_t$ is constant. The proof of Theorem \ref{unico} is done in two steps; in the first step we exclude the case $N_t=m,\ m\geq 2$, in the second one we exclude the case $N_t=\infty$. 

\begin{proposition}\label{maisque2}If $2\leq m<\infty$ then $P(N_t=m)=0$.
\end{proposition}

We will use the ideas of Newman and Schulman \cite{NS}, but note that the distribution of $\omega_t$ hasn't the finite-energy property. 

\begin{proof}[Proof of Proposition \ref{maisque2}]
Suppose that there exists $2\leq m<\infty$ such that 
\begin{equation}\label{Absurdo}
P(N_t=m)=1.
\end{equation}
For any $n\in\N$ and $l=n-1,n,n+1$, let $B_n$ be the box $[-n,n]^2$ and define the following set of bonds: $$B_{n,l}:=\{\langle x,y\rangle;x\in \partial B_n,\ y\in\partial B_l\},$$ $$I_n:=\{e\in B_{n,n};e\cap\{(n,n),(-n,n),(n,-n),(-n,-n)\}\neq\emptyset\}$$ and $$E_n:=\{e\in B_{n,n+1}; e\cap\{(n,n),(-n,n),(n,-n),(-n,-n)\}\neq\emptyset\}.$$

Define the following set of random times:

$$\mathcal{A}_n^t:=\{U\in [0,1]^\E;\mbox{all infinite clusters in }\omega_t(U)\mbox{ intersect }B_n\}.$$

At this point Newman and Schulman use the finite-energy property to open edges inside the box $B_n$, merging all infinite clusters and obtaining a contradiction. As mentioned before, our model has not the finite-energy property and we overcome this by choosing a special subset of $\mathcal{A}_n^t$, that has positive probability and forcing all edges of $\partial B_n$ to be open, merging all infinite clusters and keeping the probability of occurrence positive.

For this purpose, we define other sets of random times:

$$\mathcal{B}_n^t:=\{U\in[0,1]^{\mathbb{E}}; \omega_{t,e}(U)=1\ \forall e\in I_n\}$$ and for $\delta\in (0,1)$
$$\mathcal{C}_{\delta,n}:=\{U\in[0,1]^{\mathbb{E}}; \delta<U_e\leq 1-\delta, \forall e\in B_{n,n+1}\}.$$
Observe that $P(\mathcal{B}_n^t)>\alpha(t)$ for all $n$, where $\alpha(t)$ is a positive constant that depends only on $t$. By Equation (\ref{Absurdo}), there exists $n$ large enough such that  $P(\mathcal{A}_n^t)>1-\alpha(t)$, then it holds that 
\begin{equation}\label{AeB}
P(\mathcal{A}_n^t\cap \mathcal{B}_n^t)>0.
\end{equation} 

Given a set of bonds $A\subset B_{n,n+1}$, define
\begin{equation}\label{H}
\mathcal{H}_{n,A}^t:=\{U\in[0,1]^{\mathbb{E}}; \omega_{t,e}(U)=1,\forall e\in A\mbox{ and }\omega_{t,e}(U)=0,\forall e\in B_{n,n+1}/A\}.
\end{equation}

As $\{\mathcal{H}_{n,A}^t; A\subset B_{n,n+1}\}$ is a partition of $[0,1]^\E$, by Inequality (\ref{AeB}) there exists $A\subset B_{n,n+1}$ such that 
\begin{equation}\label{ABH}
P\left(\mathcal{A}_n^t\cap \mathcal{B}_n^t\cap\mathcal{H}^t_{n,A}\right)>0.
\end{equation}
For simplicity we write $\mathcal{D}_{n,A}^t:=\mathcal{A}^t_n\cap\mathcal{B}^t_n\cap\mathcal{H}^t_{n,A}$.

\begin{figure}
    \centering
    \begin{tikzpicture}[scale=0.5]
    \draw[line width=.03cm,gray,dashed] (-4,2)--(-3,2);
\draw(-5+1/2,2)node[above] {$e$};
\draw(-4+1/2,2)node[above] {$e^o$};
\filldraw (0,0) circle (2pt);
\draw (0.2,0.3)node {$0$};
\draw (2,-7)node[below] {$\infty$};
\draw[line width=.05cm,->] (1,-5)--(2,-5)--(2,-6)--(2,-7);
\draw (-7,-4)node[left] {$\infty$};
\draw[line width=.05cm,->] (-5,-3)--(-6,-3)--(-6,-4)--(-7,-4);
\draw (7,0)node[right] {$\infty$};
\draw[line width=.05cm,->] (5,1)--(5,0)--(6,0)--(7,0);
\draw[line width=.05cm] (-3,4)--(-3,5);
\draw[line width=.05cm] (-2,4)--(-2,5);
\draw[line width=.05cm] (-1,4)--(-1,5);
\draw[line width=.05cm] (0,4)--(0,5);
\draw[line width=.03cm,gray] (1,4)--(1,5);
\draw[line width=.03cm,gray] (2,4)--(2,5);
\draw[line width=.05cm] (3,4)--(3,5);
\draw[line width=.03cm,gray] (4,4)--(4,5);
\draw[line width=.03cm,gray] (4,-4)--(5,-4);
\draw[line width=.03cm,gray] (4,-3)--(5,-3);
\draw[line width=.05cm] (4,-2)--(5,-2);
\draw[line width=.03cm,gray] (4,-1)--(5,-1);
\draw[line width=.03cm,gray] (4,-0)--(5,0);
\draw[line width=.05cm] (4,1)--(5,1);
\draw[line width=.05cm] (4,2)--(5,2);
\draw[line width=.05cm] (4,3)--(5,3);
\draw[line width=.03cm,gray] (4,4)--(5,4);
\draw[line width=.03cm,gray] (-4,4)--(-4,5);
\draw[line width=.03cm,gray] (-3,-4)--(-3,-5);
\draw[line width=.05cm] (-2,-4)--(-2,-5);
\draw[line width=.05cm] (-1,-4)--(-1,-5);
\draw[line width=.03cm,gray] (0,-4)--(0,-5);
\draw[line width=.05cm] (1,-4)--(1,-5);
\draw[line width=.03cm,gray] (2,-4)--(2,-5);
\draw[line width=.03cm,gray] (3,-4)--(3,-5);
\draw[line width=.03cm,gray] (4,-4)--(4,-5);
\draw[line width=.05cm](-4,-4)rectangle(4,4);
\draw[line width=.03cm,gray] (-4,-4)--(-4,-5);
\draw[line width=.03cm,gray] (-4,-4)--(-5,-4);
\draw[line width=.05cm] (-4,-3)--(-5,-3);
\draw[line width=.05cm] (-4,-2)--(-5,-2);
\draw[line width=.03cm,gray] (-4,-1)--(-5,-1);
\draw[line width=.05cm] (-4,0)--(-5,0);
\draw[line width=.05cm] (-4,1)--(-5,1);
\draw[line width=.05cm] (-4,2)--(-5,2);
\draw[line width=.03cm,gray] (-4,3)--(-5,3);
\draw[line width=.03cm,gray] (-4,4)--(-5,4);
    \end{tikzpicture}\caption{The black bonds represent the set $A\subset B_{n,n+1}$ of open bonds at time $t$ in the event $\mathcal{H}_{n,A}^t$.}
		\label{ezero}
    \end{figure}
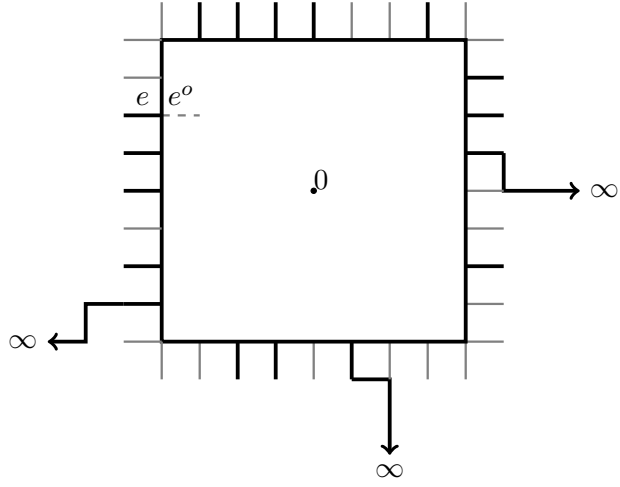

Since $\displaystyle\lim_{\delta\rightarrow0^+}P(\mathcal{C}_{\delta,n})=1$, by Equation (\ref{ABH}), there is $\delta<t$ small enough such that 
\begin{equation}\label{CD}
P(\mathcal{C}_{\delta,n}\cap\mathcal{D}_{n,A}^t)>0.
\end{equation}

If $e\in B_{n,n+1}/E_n$ where  $e=\langle x, x+u\rangle$ and $x\in B_n$, we define $e^o=\langle x,x-u\rangle$ as the {\it opposite edge of $e$} (see Figure \ref{ezero}); for $A\subset B_{n,n+1}/E_n$ we denote the sets of bonds $A^o=\{e^o\colon e\in A\}$ and $A^f=(B_{n,n+1}/A)^o\cup B_{n,n}$. Define the following set of random times on $\E(B_n)$:
$$\mathcal{E}_{n,\delta,A}:=\{U\in [0,1]^{\E(B_n)}; U_e\leq\delta,\forall e\in A^f,U_e>1-\delta, e\in \mathbb{E}(B_n)/A^f\}.$$

Given $F\subset\mathbb{E}$ (finite or not) denote by $\Pi_F$ the canonical projection from $[0,1]^{\mathbb{E}}$ onto $[0,1]^F$. Note that $\mathcal{E}_{n,\delta,A}$ can be considered a cylinder event, then,  by Equation (\ref{CD}), we have that 
{\small \begin{equation}\label{desacoplamento}
P\left(\mathcal{E}_{n,\delta,A}\times\Pi_{\mathbb{E}(B_n)^c}(\mathcal{C}_{n,A}^t\cap\mathcal{D}_{\delta,n})\right)=P\left(\mathcal{E}_{n,\delta,A}\right)P\left(\Pi_{\mathbb{E}(B_n)^c}(\mathcal{C}_{n,A}^t\cap\mathcal{D}_{\delta,n})\right)>0.
\end{equation}}

For simplicity let us denote $\mathcal{F}_{n,\delta,A}=\mathcal{E}_{n,\delta,A}\times\Pi_{\mathbb{E}(B_n)^c}(\mathcal{C}_{n,A}^t\cap\mathcal{D}_{\delta,n})$. The proof of this proposition is concluded if we show that, for each random time $U\in\mathcal{F}_{n,\delta,A}$, it holds that $\omega_t(U)$ has an unique infinite cluster. Observe that by definition of $\mathcal{E}_{n,\delta,A}$, it holds that 

\begin{equation}\label{CarasQueAbremPrimeiro}
\omega_{t,e}(U)=1, \forall e\in A^f, \forall U\in\mathcal{F}_{n,\delta,A}.
\end{equation}
 
Thus, it is enough to show that given such $U\in\mathcal{F}_{n,\delta,A}$, for all $\tilde{U}\in\mathcal{C}_{n,A}^t\cap\mathcal{D}_{\delta,n}$ with $\tilde{U}_e=U_e,\forall e\in\mathbb{E}(B_n)^c$, it holds that  
\begin{equation}\label{mudancalocal}
\omega_{t,e}(\tilde{U})=\omega_{t,e}(U),\forall e\in\mathbb{E}(B_n)^c.
\end{equation} 

Given $a,b\in[0,1]$ and $U\in[0,1]^\E$, define the set $G_{a,b}(U):=\{e\in\mathbb{E}; a<U_e\leq b\}$. 

For $t\in[0,1]$ fixed, if $e\in G_{t,1}(U)\cap\mathbb{E}(B_n)^c$, it holds that $\omega_t(U,e)=0=\omega_t(\tilde{U},e)$, yielding Equation (\ref{mudancalocal}). 

Observe that by hypothesis (\ref{Absurdo}) $t>1/2$, then it is enough to prove Equation (\ref{mudancalocal}) for $e\in G_{0,1/2}(U)\cap\mathbb{E}(B_n)^c$ and for $e\in G_{1/2,t}(U)\cap\mathbb{E}(B_n)^c$. 

Recalling that all connected components of $G_{0,1/2}(U)$ are finite, for every $e\in G_{0,1/2}(U)\cap\mathbb{E}(B_n)^c$, let $\{e_1,\dots ,e_m\}$ be its  connected component in the random graph $G_{0,1/2}(U)$ ordered in such way that $U(e_1)<U(e_2)<\cdots <U(e_m)$.

We will use induction to verify that Equation (\ref{mudancalocal}) is true for all bonds in $\{e_1,\dots ,e_m\}$.

Let us consider the initial bond. If $e_1\notin B_{n,n+1}$, it holds that $\{f\in\E ;f\sim e_1\}\subset\mathbb{E}(B_n)^c$ (from now on $f\sim e$ means that the bonds $e$ and $f$ share a common end-vertex), then $\tilde{U}(f)=U(f)$ and $U(f)>U(e_1)$ implying that $\omega_{t,e_1}(U)=\omega_{t,e_1}(\tilde{U})=1$.

If $e_1\in B_{n,n+1}$, we consider two cases. The first case is when $e_1\notin A$ (the set $A$ in the definitions of $\mathcal{H}_{n,A}^t$); for all $\tilde{U}\in\mathcal{H}_{n,A}^t$, it holds that $\omega_{t,e_1}(\tilde{U})=0$; let $x_1$ be the end-vertex of $e_1$ that belongs to $B_n$, if $x_1\in f$ and $f\neq e_1$, by Equation \ref{CarasQueAbremPrimeiro}, $\omega_{t,f}(U,)=1$ thus $\omega_{t,e_1}(U)=0$, concluding this case. The second case is $e_1\in A$; for all $\tilde{U}\in\mathcal{H}_{n,A}^t$, it holds that $\omega_{t,e_1}(\tilde{U})=1$; let $y_1$ be the end-vertex of $e_1$ that belongs to $B_{n+1}$, then we have that 
\begin{equation}\label{e_1PodeAbrirPorUmLado}
\deg\left(y_1,\omega_s(U)\right)=0,\ \forall s<U(e_1),
\end{equation}
for all $\displaystyle U\in\mathcal{E}_{n,\delta,A}\times[0,1]^{\mathbb{E}(B_n)^c}$ it holds that $U(e_1^o)>1-\delta$, then  
\begin{equation}\label{PeloOutroLadoDeE_1}
\deg(x_1,\omega_s(U))\leq 2,\ \forall s<U(e_1).
\end{equation} 
Therefore, by Equations \ref{e_1PodeAbrirPorUmLado} and \ref{PeloOutroLadoDeE_1}, $\omega_{t,e_1}(U)=1$. This proves the first step of the induction.

Suppose now that Equation (\ref{mudancalocal}) is satisfied for all bonds in $\{e_1,\dots,e_i\}$, and assume that $e_{i+1}\in B_{n,n+1}$, we consider two cases as we did for $e_1$. If $e_{i+1}\notin A$, we proceed like in the case $e_1\notin A$. If  $e_{i+1}\in A$,  let us denote $e_{i+1}=\langle x_{i+1},y_{i+1}\rangle$, with $x_{i+1}\in B_n$ and $y_{i+1}\in B_{n+1}$; for all $\displaystyle U\in\mathcal{E}_{n,\delta,A}\times[0,1]^{\mathbb{E}(B_n)^c}$ it holds that $U(e_{i+1}^o)>1-\delta$, then  

\begin{equation}\label{umlado1}
\deg(x_{i+1},\omega_s(U))\leq 2,\ \forall s<U(e_{i+1}).
\end{equation} 
For all $\tilde{U}\in\mathcal{H}_{n,A}^t$, it holds that $e_{i+1}$ is open in $\omega_{t}(\tilde{U})$, then there exists a bond $f$ with $y_{i+1}\in f$ such that $\omega_{t,f}(\tilde{U})=0$. If $f\in\{e_1,\ldots,e_i\}$, by induction hypothesis,  $\omega_{t,f}(U)=\omega_{t,f}(\tilde{U})=0$; if $f\notin\{e_1,\ldots,e_i\}$, we have that $U(f)>U(e_{i+1})$; anyway it holds that
\begin{equation}\label{outrolado1}
\deg(y_{i+1},\omega_s(U))\leq 2,\ \forall s<U(e_{i+1}).
\end{equation}
Then, by Equations (\ref{umlado1}) and (\ref{outrolado1}), $\omega_{t,e_{i+1}}(U)=\omega_{t,e_{i+1}}(\tilde{U})=1$. 

Now assume that $e_{i+1}\notin B_{n,n+1}$. Consider the edges  $\{f\in\E; f\sim e_{i+1}\}\subset\mathbb{E}(B_n)^c$. If $U(f)<U(e_{i+1})$, then $f\in\{e_1,\dots,e_i\}$, and Equation \eqref{mudancalocal} holds for $f$. If $U(f)>U(e_{i+1})$, then $\omega_{s,f}(U)=\omega_{s,f}(\tilde{U})=0$ for $s< U(e_{i+1})$.
We conclude that for any endvertex of $e_{i+1}$, say $x$, we must have 

$$\deg(x,\omega_s(U))=\deg(x,\omega_s(\tilde{U})),\;\; \forall s<U(e_{i+1}),$$
therefore $\omega_{t,e_{i+1}}(U)=\omega_{t,e_{i+1}}(\tilde{U})$.

This proves Equation (\ref{mudancalocal}) for all $e\in G_{0,1/2}(U)\cap\mathbb{E}(B_n)^c$.

%

Now, let   $\{e_1,\dots ,e_m\}$ be a connected component in the random graph $G_{1/2,t}(U)$ (it is also finite because $P(e_i\in G_{1/2,t}(U))\leq\frac{1}{2},\forall i$),  such that $U(e_1)<U(e_2)<\cdots <U(e_m)$. To prove that Equation (\ref{mudancalocal}) holds for the edge $e_{i}$, we proceed as in the preceeding induction step using the fact that Equation (\ref{mudancalocal}) holds for $\{e_1,\dots,e_{i-1}\}\cup G_{1/2,t}(U)$. 

This concludes the proof of Proposition \ref{maisque2}.


\comment{
If the initial bond $e_1\notin B_{n,n+1}$, it holds that $\{f\in\E ;f\sim e_1\}\subset\mathbb{E}(B_n)^c$. Suppose that there is an end-vertex of $e_1=\langle x_1,y_1\rangle$, let's say $x_1$, such that all of the three edges sharing $x_1$ with $e_1$ belong to $G_{0,1/2}(U)$ and are open in $\omega_t(U)$. In this case $\omega_{t,e_1}(U)=0$. As we have already shown that all edge in $G_{0,1/2}(U)$ satisfies the Equation \ref{mudancalocal}, then these three edges are also open in $\omega_t(\tilde{U})$, thus $\omega_{t,e_1}(\tilde{U})=\omega_{t,e_1}(U)=0$. Otherwise $\deg(u,\omega_{1/2}(U))=\deg(u,\omega_{1/2}(\tilde{U}))\leq 2,\forall u\in\{x_1,y_1\}$, since $U(e_1)=\min\{U(e_i); i=1,\cdots,m\}>1/2$, it follows that $\omega_{t,e_1}(U)=\omega_{t,e_1}(\tilde{U})=1$.

If the initial bond $e_1\in B_{n,n+1}$, as before we consider two cases depending on $e_1$ belongs or not to the set $A$. If $e_1\notin A$, we follow the same steps when $e_1\in G_{0,1/2}(U)$. If $e_1\in A$, it holds that $\omega_{t,e_1}(\tilde{U})=1,\forall \tilde{U}\in\mathcal{H}_{n,A}^t$; let $y_1$ be the end-vertex of $e_1$ that belongs to $B_{n+1}$, then there exists a bond $f$, with $y_1\in f$ and $\omega_{t,f}(\tilde{U})=0$; if $f\in G_{0,1/2}(U)$ then we have already shown that $\omega_t(U,f)=\omega_t(\tilde{U},f)$, then $f$ is also closed in $\omega_t(U)$; if $f\in G_{1/2,t}$ then $U(f)>U(e_1)$. Anyway, we have that 
\begin{equation}\label{umlado2}
\deg(y_1,\omega_s(U))\leq 2, \forall s<U(e_1).
\end{equation} 
Since $U\in\mathcal{E}_{n,\delta,A}$, it holds that $U(e_{1}^o)>1-\delta$, then 
\begin{equation}\label{outrolado2}
\deg(x_{1},\omega_s(U))\leq 2, \forall s<U(e_{1}).
\end{equation}

Therefore, by Equations \ref{umlado2} and \ref{outrolado2}, $\omega_t(U,e_1)=\omega_t(\tilde{U},e_1)=1$. This prove the first step of the induction. The inductions step is done exactly as in the case $e\in G_{1/2,t}(U)\cap\mathbb{E}(B_n)^c$, this finishes the proof of Equation \ref{mudancalocal} and conclude the proof of Proposition \ref{maisque2}.}


\end{proof}

The next proposition excludes the case $N_t=\infty$.

\begin{proposition}\label{infinito} $$P(N_t=\infty)=0.$$
\end{proposition}

The idea is to use an improvement of the Burton-Keane's technique \cite{BK} that can be found in Section 7.3 of \cite{LP} or originally in \cite{BLPS}, but note that the distribution of $\omega_t$ is not {\it insertion tolerant} ({\it positive finite energy} in the sense of \cite{GKN}). 

\begin{proof}[Proof of Proposition \ref{infinito}] Suppose, by contradiction, that the event $\{N_t=\infty\}$ has probability one. We modify slightly the previous definition of $\mathcal{A}_n^t$. Define the following sets of random times
$$\bar{\mathcal{A}}_n^t:=\{U\in [0,1]^\E;\mbox{at least three infinite clusters in }\omega_t(U)\mbox{ intersect }B_n\},$$
the sets $\mathcal{B}_n^t$, $\mathcal{C}_{\delta,n}$, $\mathcal{E}_{n,\delta,A}$ and $\mathcal{H}_{n,A}^t$ are defined like in Proposition \ref{maisque2}, $\bar{\mathcal{D}}_{n,A}^t:=\bar{\mathcal{A}}^t_n\cap\mathcal{B}^t_n\cap\mathcal{H}^t_{n,A}$ and $\bar{\mathcal{F}}_{n,\delta,A}:=\mathcal{E}_{n,\delta,A}\times\Pi_{\mathbb{E}(B_n)^c}(\mathcal{C}_{n,A}^t\cap\bar{\mathcal{D}}_{\delta,n}).$

Like in Proposition \ref{maisque2} there exists $A\subset B_{n,n+1}$, such that for any $n$ large enough and $\delta$ close to zero, it holds that $P(\bar{\mathcal{F}}_{n,\delta,A})>0$.

Observe that for any $U\in\bar{\mathcal{F}}_{n,\delta,A}$, the configuration $\omega_t(U)$ has a cluster with at least three ends (the number of ends is the supremum, over all finite subsets $K$ of $\E$, of the number of infinite connected compounds of $\omega_t(U) /K$), the cluster of any vertex in $\partial B_n$ has this property.

The fact $P(\bar{\mathcal{F}}_{n,\delta,A})>0$ ensures that we are under the hypothesis of Lemma 7.7 of \cite{LP}. Therefore, there is (on a larger probability space) a random forest $\mathfrak{F}\subset\omega_t(U)$ such that

(a) the distribution of the pair $(\omega_t(U),\mathfrak{F})$ is translation-invariant;

(b) the set of forests $\mathfrak{F}$ that has a component with at least three ends has positive probability.

Let $\mathbb{P}_t$ be the distribution of the pair $(\omega_t(U),\mathfrak{F})$ and $\mathbb{E}_t$ your respective expectation.  

A vertex $v$ is a {\it furcation} of a configuration $\omega$ if closing all edges incident to $v$, the cluster of $v$ splits in at least three infinite clusters. 

Let $X$ be set of the furcations of $\mathfrak{F}$. By (b), it holds that $\mathbb{P}_t(v\in X)>0$ for any $v\in B_n$. By translation invariance, we have that  
\begin{equation}\label{Contradiction}
\mathbb{E}_t[\#(X\cap B_n)]=\#(B_n)\mathbb{P}_t\left(0\in X\right).
\end{equation}
 
It is well known that $\#(\partial B_n)\geq \#(B_n\cap X)$, taking the expectation and using Equation \ref{Contradiction}, we obtain that $\mathbb {L}^d$ is not amenable, a contradiction. This concludes the proof of this proposition.
\end{proof}

Combining Propositions \ref{maisque2} and \ref{infinito}, the proof of Theorem \ref{unico} is concluded.

\end{proof}

\section{Constrained-degree percolation model on regular trees}

For all $d\in\N$, we denote $[d]= \{1,\ldots, d\}$, and define the set
\begin{align*}
[d]_\star = \bigcup_{0\leq n < \infty} [d]^n;
\end{align*}
where $[d]^0$ is understood to consist of a single point $o$, the root of the tree.
Points of $[d]^n$ are represented as $x = (x_1,\ldots, x_n)$, for $a\in [d]$, we define the concatenation $x\cdot a = (x_1,\ldots, x_n,a).$

The regular $d$-ary tree is the graph $\T_{d}=(\V_d,\E_d)$, where $\V_{d} = [d]_\star$ and $\E_{d}=\{\langle x, x\cdot a\rangle; \;x \in \V_{d},\;a\in [d]\}$.

Observe that exactly the same proof of Proposition \ref{k2} works for the Constrained-degree percolation model on $\T_d$. Thus, we have the following statement:  

\begin{proposition}For the Constrained-degree percolation model on the regular $d$-ary tree $\mathbb{T}^d,\ d\geq 2$, it holds that $\theta^{\mathbb{T}^d,(2)}(t)=0,\ \forall t\in[0,1].$
\end{proposition} 

The next result shows that, with the restriction $k_v=3,\forall v\in\T_d$, there is percolation on $\T^d$. 

\begin{theorem}\label{tree}For the Constrained-degree percolation model on $\T_d,\ d\geq2$, it holds that $t_c(\mathbb{T}^d,(3))<1$.
\end{theorem}
\begin{proof}There is nothing to do when $d=2$. Given the sequence of uniform random times $U\in[0,1]^{\E_{d}}$, let $\mathfrak{F}=\mathfrak{F}(U)$ be the random forest  $\mathfrak{F}=(\V_d, \E (\mathfrak{F}))$ where
\begin{align*}
\E (\mathfrak{F}) = \{e=\langle x,x\cdot a\rangle\in\E_d; \#\{b\in [d]\backslash\{a\}; U(e)>U(\langle x,x\cdot b\rangle)\}\leq 1\}.
\end{align*}

In words, $e\in\E (\mathfrak{F})$ if $e$ has at most one brother whose clocks ring before $U(e)$. Note that $\mathfrak{F}$ is a collection of binary trees, let ${\cal T}$ be the one that contains the origin.

Given any bond $e=\langle x,x\cdot a\rangle\in \E({\cal T})$, we declare this bond $e$ as {\em red} if
\begin{align*}
\#\{b\in [d]; U(e)>U(\langle x\cdot a,x\cdot a\cdot b\rangle)\}\leq 2.
\end{align*}

That is, $e\in\E ({\cal T})$ is red if and only if $e$ has at most two sons in $\T_d$ whose clocks ring before $U(e)$. Note that if $e$ is red then $e$ will open at time $U_e$ in the Constrained-degree percolation model with $k=3$. Thus to conclude this proof, it is enough to prove that there is percolation of red bonds on ${\cal T}$.

Let $e_n^{(1)}=\langle x,x\cdot a\rangle , e_n^{(2)}=\langle x,x\cdot b\rangle\in\E ({\cal T})$ be the bonds of $n$-th generation of ${\cal T}$, i.e. $x\in [d]^{n-1}$ and $a,b\in[d]$, with $U(e_1^n)<U(e_2^n)$.

Given $e\in\E ({\cal T})$, it holds that 
\begin{align}\label{arvred}
\mathbb{P}(e \mbox{ is red})=\mathbb{P}(e \mbox{ is red},e=e_n^{(1)})+\mathbb{P}(e \mbox{ is red},e=e_n^{(2)}) \nonumber\\
=\frac{1}{2}\mathbb{P}(e \mbox{ is red}|e=e_n^{(1)}) + \frac{1}{2}\mathbb{P}(e \mbox{ is red}|e=e_n^{(2)})
\end{align}

Now, observe that 
\begin{equation}\nonumber \begin{array}{l}\mathbb{P}( e \mbox{ is red}|e=e_n^{(1)})=d\left[\frac{1}{2d}+\frac{d(2d-2)!}{(2d)!}+\frac{d(d-1)(2d-3)!}{(2d)!}\right],\\
\mathbb{P}(e \mbox{ is red}|e=e_n^{(2)})=d\left[\frac{(d-1)(2d-2)!}{(2d)!}+\frac{2d(d-1)(2d-3)!}{(2d)!}\right.\\
\left.+\frac{3d(d-1)^2(2d-4)!}{(2d)!}\right].\end{array}
\end{equation}

Combining the equations above with (\ref{arvred}), a simple calculation shows that $\mathbb{P}(e \mbox{ is red})>\frac{1}{2}$. Since ${\cal T}$ is a binary tree, the red edges have the distribution of a supercritical branching processes. This yields percolation of red bonds on ${\cal T}$.
\end{proof}

\section*{Acknowledgements}
The research of B.N.B.L. and R.S. are partially supported by CAPES, FAPEMIG (Programa Pesquisador Mineiro) and CNPq. D.C.S. and R.T would like to thank CAPES and CNPq, respectively.


\begin{thebibliography}{999}

\bibitem{AG} Aizenman M., Grimmett G.R., \emph{Strict monotonicity for critical points in percolation and ferromagnetic models}, Journal of Statistical Physics {\bf 63}, 817-835, (1991).
\bibitem{AKN} Aizenman M., Kesten H., Newman C., \emph{Uniqueness of the infinite cluster and continuity of connectivity functions for short and long range percolation}, Communications in Mathematical Physics {\bf 111}, 505-531, (1987).
\bibitem{BBR} Balister P.,Bollob\'as B, Riordan O., \emph{Essential enhancements revisited}, arXiv:1402.0834, (2014).
\bibitem{Be} Benjamini I., Private communication.
\bibitem{BLPS} Benjamini I., Lyons R., Peres Y. Schramm O., \emph{Group-invariant Percolation on Graphs}, Geometric \& Functional Analysis GAFA {\bf 9}, 29-66, (1999).
\bibitem{BR} Bollob\'as B, Riordan O., \emph{Percolation}, Cambridge University Press, (2006).
\bibitem{BK} Burton R., Keane M., \emph{Density and uniqueness in percolation}, Communications in Mathematical Physics {\bf 121}, 501-505, (1989).
\bibitem{GKN} Gandolfi A., Keane M., Newman C., \emph{Uniqueness of the infinite component in a random graph with applications to percolation and spin glasses}, Probability Theory and Related Fields {\bf 92}, 511-527, (1992).
\bibitem{Gr} Grimmett G., \emph{Percolation}, 2nd edition, Springer-Verlag, New York, (1999).
\bibitem{GJ} Grimmett G., Janson S.,\emph{Random graphs with forbidden vertex degrees}, Random Structures and Algorithms {\bf 37}, 137–175, (2010).
\bibitem{GL} Grimmett G., Li Z.,\emph{The 1-2 model}, Contemporary Mathematics {\bf 969}, 139–152, (2017).
\bibitem{GMM} Garet O., Marchand R., Marcovici I. \emph{Does Eulerian percolation on $\mathbb{Z}^2$ percolate?} ALEA, Lat. Am. J. Probab. Math. Stat. {\bf 15}, 279–294 (2018)
\bibitem{HJ} Häggström O., Jonasson J., \emph{Uniqueness and non-uniqueness in percolation theory}, Probability Surveys {\bf 3}, 389-344, (2006).
\bibitem{Ha} Harris T., \emph{A lower bound for the critical probability in a certain percolation process}, In Mathematical Proceedings of the Cambridge Philosophical Society {\bf 56}, 13-20, (1960).
\bibitem{HL} Holroyd A.E., Li Z, \emph{Constrained percolation in two dimensions}, arXiv:1510.03943v2, (2016).
\bibitem{Ke} Kesten H., \emph{Percolation Theory for Mathematicians}, Birkh\"auser, Boston, (1982).
\bibitem{LP} Lyons R., Peres Y., \emph{Probability on Trees and Networks}, Cambridge University Press, (2017). 
\bibitem{NS} Newman C., Schulman L., \emph{Infinite Clusters in Percolation Models}, Journal of Statistical Physics {\bf 26}, 613-628, (1981).
\bibitem{Te} Teodoro R., \em{Constrained-degree Percolation}, PhD thesis, IMPA, (2014).















\end{thebibliography}
\end{document}